\documentclass{article}

\usepackage[utf8]{inputenc}
\usepackage{amssymb, amsmath, amsthm}
\usepackage[backend=bibtex,firstinits=true]{biblatex}
\usepackage{hyperref}

\DeclareMathOperator{\Aut}{Aut}
\DeclareMathOperator{\Hom}{Hom}
\DeclareMathOperator{\End}{End}
\DeclareMathOperator{\ev}{ev}
\DeclareMathOperator{\diag}{diag}
\DeclareMathOperator{\disc}{disc}
\DeclareMathOperator{\Contr}{Contr}

\newcommand{\Mbar}{\overline{M}}
\newcommand{\CC}{\mathbb C}
\newcommand{\QQ}{\mathbb Q}
\newcommand{\ZZ}{\mathbb Z}
\newcommand{\PP}{\mathbb P}
\newcommand{\strata}{\mathcal S}

\newtheorem{lem}{Lemma}
\newtheorem{prop}{Proposition}
\newtheorem*{conj}{Reconstruction Conjecture}
\newtheorem{thm}{Theorem}
\newtheorem{rthm}{Theorem}

\theoremstyle{definition}
\newtheorem{defi}{Definition}
\newtheorem{ex}{Example}

\theoremstyle{remark}
\newtheorem{rem}{Remark}

\begin{document}

\title{Comparing tautological relations from the equivariant
  Gromov-Witten theory of projective spaces and spin structures}
\author{Felix Janda}

\maketitle
\abstract{Pandharipande-Pixton-Zvonkine's proof of Pixton's
  generalized Faber-Zagier relations in the tautological ring of
  $\Mbar_{g, n}$ has started the study of tautological relations from
  semisimple cohomological field theories. In this article we compare
  the relations obtained in the examples of the equivariant
  Gromov-Witten theory of projective spaces and of spin structures. We
  prove an equivalence between the $\PP^1$- and $3$-spin relations,
  and more generally between restricted $\PP^m$-relations and
  similarly restricted $(m + 2)$-spin relations. We also show that the
  general $\PP^m$-relations imply the $(m + 2)$-spin relations.}

\section{Introduction}

The study of the Chow ring of the moduli space of curves was initiated
Mumford in \cite{Mu83}. Because it is difficult to understand the
whole Chow ring in general, the tautological subrings of classes
reflecting the geometry of the objects parametrized by the moduli
space were introduced. The tautological ring $R^*(\Mbar_{g, n})$ is
compactly described \cite{FaPa05} as the smallest system
\begin{equation*}
  R^*(\overline M_{g, n}) \subseteq A^*(\overline M_{g, n})
\end{equation*}
of subrings compatible with push-forward under the tautological maps,
i.e. the maps obtained from forgetting marked points or gluing curves
along common markings.

There is a canonical set of generators parametrized by decorated
graphs \cite{GrPa03}. The formal vector space $\strata_{g, n}$
generated by them, the \emph{strata algebra}, therefore admits a
surjective map to $R^*(\Mbar_{g, n})$ and the structure of the
tautological ring is determined by the kernel of this
surjection. Elements of the kernel are called \emph{tautological
  relations}.

In \cite{Pi12P} A.~Pixton proposed a set of (at the time conjectural)
relations generalizing the relations of Faber-Zagier in $R^*(M_g)$.
Furthermore, he conjectured that these give all tautological
relations. The first proof \cite{PPZ15} of the fact that the
conjectural relations are actual relations (in cohomology) brought
cohomological field theories (CohFTs) into the picture.

A CohFT on a free module $V$ of finite rank over a base ring $A$ is a
system of classes $\Omega_{g, n}$ behaving nicely under pull-back via
the tautological maps. A CohFT can also be used to give $V$ the
structure of a Frobenius algebra. The CohFT is called semisimple if,
after possible base extension, the algebra $V$ has a basis of
orthogonal idempotents.

For semisimple CohFTs there is a conjecture by Givental \cite{Gi01b}
proven in some cases by himself and in full generality in cohomology
by Teleman \cite{Te12}, giving a reconstruction of the CohFT from its
genus $0$, codimension $0$ part and the data of a power series $R(z)$
of endomorphisms of $V$. The formula naturally lifts to the strata
algebra.

To get relations from a semisimple cohomological field theory we can
use that the reconstructed CohFT of elements in the strata algebra is
in general only defined over an extension $B \leftarrow A$. However
since we have started out with a CohFT over $A$, this implies that
certain linear combinations of elements in the strata algebra have to
vanish under the projection to the tautological ring.

This procedure was essentially used in the proof \cite{PPZ15} in the
special example of the CohFT defined from Witten's $3$-spin
class. There the base ring is a polynomial ring in one variable but
the reconstructed CohFT seems to have poles in this variable.

In~\cite{rspin} (in preparation) the authors construct tautological
relations using Witten's $r$-spin class for any~$r \geq 3$. Given a
list of integers $a_1, \dots, a_n \in \{0, \dots, r-2\}$, Witten's
class $W_{g,n}(a_1, \dots, a_n)$ is a cohomology class on
$\Mbar_{g,n}$ of pure degree
\begin{equation*}
  D_{g,n}(a_1, \dots, a_n)
  =\frac{(r-2)(g-1) + \sum_{i=1}^n a_i}r.
\end{equation*}
Witten's class can be ``shifted'' by any vector in the vector space
$\langle e_0, \dots,\linebreak[1] e_{r-2} \rangle$ to obtain a
semisimple CohFT. In practice, the authors use two particular shifts
for which the answer can be explicitly computed. Shifted Witten's
class is of mixed degree: more precisely, the degrees of its
components go from 0 to $D_{g,n}(a_1, \dots, a_n)$. On the other hand,
the Givental-Teleman classificiation of semisimple CohFTs gives an
expression of the shifted Witten class in terms of tautological
classes. The authors conclude that the components of this expression
beyond degree $D_{g,n}(a_1, \dots, a_n)$ are tautological relations.

\medskip

This article studies how relations from spin structures are related to
the relations obtained from the CohFT defined from the equivariant
Gromov-Witten theory of projective spaces. The following two theorems
are our main results.

\begin{rthm}[rough version]
  The relations obtained from the equivariant Gromov-Witten theory of
  $\PP^m$ imply the $(m + 2)$-spin relations.
\end{rthm}
\begin{rthm}[rough version]
  A special restricted set of relations from equivariant $\PP^m$ is
  equivalent to a corresponding restricted set of $(m + 2)$-spin
  relations. For $\PP^1$ and 3-spin no restriction is necessary.
\end{rthm}
Since for equivariant $\PP^m$ the reconstruction holds in Chow,
Theorem~\ref{thm:limit} implies that the higher spin relations also
hold in Chow.

We will give strong evidence that the method of proof for
Theorem~\ref{thm:equiv} cannot directly be extended to an equivalence
between the full $\PP^m$- and $(m + 2)$-spin relations for $m >
2$. Possibly, there are more $\PP^m$- than $(m + 2)$-spin relations.

Any of the theorems gives another proof of the fact that Pixton's
relations hold in Chow. In fact, the proof of Theorem~\ref{thm:limit} in
the case $m = 1$ is essentially a simplified version of the author's
previous proof in \cite{Ja13P}.

This article does not give a comparison between relations from CohFTs
of different dimensions, nor does it consider all relations from
equivariant $\mathbb P^m$. On the other hand, if indeed Pixton's
relations are all tautological relations, the $3$-spin relations have
to imply the relations from any other semisimple CohFT. Yet, for
example it not clear how the $4$-spin relations can be written in
terms of $3$-spin relations.

\medskip

The article is structured as follows. In Section~\ref{sec:cohft} we
give definitions of CohFTs, discuss the $R$-matrix action on CohFTs
and the reconstruction result. We then in Section~\ref{sec:cohft:two}
turn to the two examples of equivariant $\PP^m$ and the CohFT from the
$A_{m + 1}$-singularity. In Section~\ref{sec:cohft:rels} we describe
the general procedure of obtaining relations from semisimple CohFTs
and general methods of proving that the relations from one CohFT imply
the relations from another. We then state precise versions of
Theorem~\ref{thm:limit} and \ref{thm:equiv}. Section~\ref{sec:osc}
discusses explicit expression of the $R$-matrices in both theories in
terms of asymptotics of oscillating integrals. The constraints
following from these expressions will be used in the next sections. We
also note a connection to Airy functions. Section~\ref{sec:limit} and
Section~\ref{sec:equiv} give proofs of Theorem~\ref{thm:limit} and
\ref{thm:equiv}. Finally, Section~\ref{sec:noequiv} gives evidence
why, with the methods used in the proofs of the theorems, an
equivalence between $\PP^m$- and $(m + 2)$-spin relations cannot be
established. Since the reconstruction result of Givental we use to get
relations in Chow has never appeared explicitly in the literature, we
recall its proof in Appendix~\ref{sec:givenloc}.

\subsection*{Acknowledgments}

This work started out of discussions at the conference
\emph{Cohomology of the moduli space of curves} organized by the
\emph{Forschungsinstitut für Mathematik} at ETH Zürich. The author is
especially grateful for encouragement and explanations from
Y.P. Lee. The author thanks the \emph{Mathematics Department of the
  University of Utah} for the hospitality during his visit in January
2014. The author is grateful for comments of A. Pixton and D. Zvonkine
on earlier versions of the paper, a copy of a draft of \cite{rspin}
from R. Pandharipande and for helpful discussions with S. Keel.

The author was supported by the Swiss National Science Foundation
grant SNF 200021\_143274.

\section{Cohomological field theories} \label{sec:cohft}

\subsection{Definitions}

Cohomological field theories were first introduced by Kontsevich and
Manin in \cite{KoMa97} to formalize the structure of classes from
GW-theory. Let $A$ be an integral, commutative $\QQ$-algebra, $V$ a
free $A$-module of finite rank and $\eta$ a non-singular bilinear form
on $V$.
\begin{defi}
  A cohomological field theory (CohFT) $\Omega$ on $(V, \eta)$ is a
  system
  \begin{equation*}
    \Omega_{g, n} \in A^*(\Mbar_{g, n}) \otimes_{\mathbb Q} (V^*)^{\otimes n}
  \end{equation*}
  of multilinear forms with values in the Chow ring of $\Mbar_{g, n}$
  satisfying the following properties:
  \begin{description}
  \item[Symmetry] $\Omega_{g, n}$ is symmetric in its $n$ arguments
  \item[Gluing] The pull-back of $\Omega_{g, n}$ via the gluing map
    \begin{equation*}
      \Mbar_{g_1, n_1 + 1} \times \Mbar_{g_2, n_2 + 1} \to \Mbar_{g, n}
    \end{equation*}
    is given by the direct product of $\Omega_{g_1, n_2 + 1}$ and
    $\Omega_{g_2, n_2 + 1}$ with the bivector $\eta^{-1}$ inserted at
    the two gluing points. Similarly for the gluing map $\Mbar_{g - 1,
      n + 2} \to \Mbar_{g, n}$ the pull-back of $\Omega_{g, n}$ is given
    by $\Omega_{g - 1, n + 2}$ with $\eta^{-1}$ inserted at the two
    gluing points.
  \item[Unit] There is a special element $\mathbf 1 \in V$ called the
    \emph{unit} such that
    \begin{equation*}
      \Omega_{g, n + 1}(v_1, \dotsc, v_n, \mathbf 1)
    \end{equation*}
    is the pull-back of $\Omega_{g, n}(v_1, \dotsc, v_n)$ under the
    forgetful map and
    \begin{equation*}
      \Omega_{0, 3}(v, w, \mathbf 1) = \eta(v, w).
    \end{equation*}
  \end{description}
\end{defi}
\begin{defi}
  The quantum product $(u, v) \mapsto uv$ on $V$ with unit $\mathbf 1$
  is defined by the condition
  \begin{equation} \label{eq:defquantumproduct}
    \eta(uv, w) = \Omega_{0, 3} (u, v, w).
  \end{equation}
\end{defi}
\begin{defi}
  A CohFT is called semisimple if there is a base extension $A \to B$
  such that the algebra $V \otimes_A B$ is semisimple.
\end{defi}

\subsection{First Examples}

\begin{ex}
  \label{ex:trivial}
  For each Frobenius algebra there is the \emph{trivial CohFT} (also
  called \emph{topological field theory} or TQFT) $\Omega_{g, n}$
  characterized by \eqref{eq:defquantumproduct} and that
  \begin{equation*}
    \Omega_{g, n} \in A^0 (\Mbar_{g, n}) \otimes (V^*)^{\otimes n}.
  \end{equation*}
  Let us record an explicit formula for Appendix~\ref{sec:givenloc}:
  In the case that the Frobenius algebra is semisimple, there is a
  basis $\epsilon_i$ of orthogonal idempotents of $V$ and
  \begin{equation*}
    \tilde\epsilon_i = \frac{\epsilon_i}{\sqrt{\Delta_i}},
  \end{equation*}
  where $\Delta_i^{-1} = \eta(\epsilon_i, \epsilon_i)$, is a
  corresponding orthonormal basis of \emph{normalized idempotents}. We
  have
  \begin{equation*}
    \Omega_{g, n}(\tilde\epsilon_{i_1}, \dotsc, \tilde\epsilon_{i_n}) =
    \begin{cases}
      \sum_j \Delta_{i_j}^{g - 1}, & \text{if } n = 0, \\
      \Delta_{i_1}^{\frac{2g - 2 + n}2}, & \text{if } i_1 = \dotsb = i_n, \\
      0, & \text{else.}
    \end{cases}
  \end{equation*}
\end{ex}

\begin{ex}
  \label{ex:hodge}
  The Chern polynomial $c_t(\mathbb E)$ of the Hodge bundle $\mathbb
  E$ gives a $1$-dimensional CohFT over $\QQ[t]$.
\end{ex}

\begin{ex}
  Let $X$ be a smooth, projective variety such that the cycle class
  map gives an isomorphism between Chow and cohomology rings. Let $A =
  \QQ[\![q^\beta]\!]$ be its Novikov ring. Then the Gromov-Witten
  theory of $X$ defines a CohFT based on the $A$-module $A^*(X)
  \otimes A$ by the definition
  \begin{equation*}
    \Omega_{g, n} (v_1, \dotsc, v_n)
    = \sum_{\beta} \pi_*\left(\prod_{i = 1}^n \ev_i^*(v_i) \cap [\Mbar_{g, n}(X, \beta)]^{vir}\right) q^\beta,
  \end{equation*}
  where the sum ranges over effective, integral curve classes, $\ev_i$
  is the $i$-th evaluation map and $\pi$ is the forgetful map $\pi:
  \Mbar_{g, n}(X, \beta) \to \Mbar_{g, n}$. The gluing property
  follows from the splitting axiom of virtual fundamental classes. The
  fundamental class of $X$ is the unit of the CohFT and the unit
  axioms follow from the identity axiom in GW-theory.

  For a torus action on $X$, this example can be enhanced to give a
  CohFT from the equivariant GW-theory of $X$.
\end{ex}

\subsection{The \texorpdfstring{$R$}{R}-matrix
  action} \label{sec:Rmatrixaction}

\begin{defi} \label{def:slg}
  The (upper part of the) symplectic loop group is defined as the
  subgroup of the group of endomorphism valued power series $R = 1 +
  O(z)$ in $z$ satisfying the symplectic condition
  \begin{equation*}
    \eta(R(z)v, R(-z)w) = \eta(v, w)
  \end{equation*}
  for all vectors $v$ and $w$.
\end{defi}

An action of this group on the space of CohFTs makes it interesting
for us. In its definition the endomorphism valued power series $R$ is
evaluated at cotangent line classes and applied to vectors.

Given a CohFT $\Omega_{g, n}$ the new CohFT $R\Omega_{g, n}$ takes the
form of a sum over dual graphs $\Gamma$
\begin{equation*}
  R\Omega_{g, n}(v_1, \dotsc, v_n)
  = \sum_{\Gamma} \frac 1{\Aut(\Gamma)} \xi_*\left(\prod_v \sum_{k = 0}^\infty \frac 1{k!} \varepsilon_* \Omega_{g_v, n_v + k}(\dots)\right),
\end{equation*}
where $\xi: \prod_v \Mbar_{g_v, n_v} \to \Mbar_{g, n}$ is the gluing
map of curves of topological type $\Gamma$ from their irreducible
components, $\varepsilon: \Mbar_{g_v, n_v + k} \to \Mbar_{g_v, n_v}$
forgets the last $k$ markings and we still need to specify what is put
into the arguments of $\prod_v \Omega_{g_v, n_v + k_v}$.
\begin{itemize}
\item Into each argument corresponding to a marking of the curve, put
  $R^{-1}(\psi)$ applied to the corresponding vector.
\item Into each pair of arguments corresponding to an edge put the
  bivector
  \begin{equation*}
    \frac{R^{-1}(\psi_1) \eta^{-1} R^{-1}(\psi_2)^t - \eta^{-1}}{-\psi_1 - \psi_2} \in \Hom(V^*, V)[\![\psi_1, \psi_2]\!] \cong V^{\otimes 2}[\![\psi_1, \psi_2]\!],
  \end{equation*}
  where one has to substitute the $\psi$-classes at each side of the
  normalization of the node for $\psi_1$ and $\psi_2$. By the
  symplectic condition this is well-defined.
\item Into each of the additional arguments for each vertex put
  \begin{equation*}
    T(\psi) := \psi(1 - R^{-1}(\psi)) \mathbf 1,
  \end{equation*}
  where $\psi$ is the cotangent line class corresponding to that
  vertex. Since $T(z) = O(z^2)$ the above $k$-sum is finite.
\end{itemize}

\begin{conj}[Givental]
  The $R$-matrix action is free and transitive on the space of
  semisimple CohFTs based on a given Frobenius algebra.
\end{conj}
\setcounter{thm}{2}
\begin{thm}[Givental\cite{Gi01b}] \label{thm:reconstr:toric}
  Reconstruction for the equivariant GW-theory of toric targets holds
  in Chow.
\end{thm}
\begin{thm}[Teleman\cite{Te12}]
  Reconstruction holds in cohomology.
\end{thm}
\begin{rem}
  Givental's original conjecture was only stated in terms of the
  descendent integrals of the CohFT and there is no explicit proof of
  Theorem~\ref{thm:reconstr:toric} in the literature. Therefore in
  Appendix~\ref{sec:givenloc} we recall the well-known lift of
  Givental's proof to CohFTs.
\end{rem}

\begin{ex}
  \label{ex:mumford}
  By Mumford's Grothendieck-Riemann-Roch calculation \cite{Mu83} the
  single entry of the $R$-matrix taking the trivial one-dimensional
  CohFT to the CohFT from Example~\ref{ex:hodge} is given by
  \begin{equation*}
    \exp\left(\sum_{i = 1}^\infty \frac{B_{2i}}{2i(2i - 1)} (tz)^{2i - 1}\right),
  \end{equation*}
  where $B_{2i}$ are the Bernoulli numbers, defined by
  \begin{equation*}
    \sum_{i = 0}^\infty B_i \frac{x^i}{i!} = \frac x{e^x - 1}.
  \end{equation*}
  More generally, if we consider a more general CohFT given by a
  product of Chern polynomials (in different variables) of the Hodge
  bundle, the $R$-matrix from the trivial CohFT is the product of the
  $R$-matrices of the factors.
\end{ex}

\subsection{Frobenius manifolds and the quantum differential
  equation} \label{sec:cohft:frobenius}

There is a natural way to deform a CohFT $\Omega_{g, n}$ on $V$ over
$A$ to a CohFT over $A[\![V]\!]$. For a basis $\{e_\mu\}$ of $V$ let
\begin{equation*}
  p = \sum t^\mu e_\mu
\end{equation*}
be a formal point on $V$. Then the deformed CohFT is given by
\begin{equation*}
  \Omega_{g, n}^p (v_1, \dotsc, v_n)
  = \sum_{k = 0}^\infty \frac 1{k!} \pi_*\Omega_{g, n + k}(v_1, \dotsc, v_n, p, \dotsc, p).
\end{equation*}
Notice that the deformation is constant in the direction of the unit.

The quantum product on the deformed CohFT gives $V$ the structure of a
(formal) \emph{Frobenius manifold} \cite{Du94}. The $e_\mu$ induce
flat vector fields on $V$ corresponding to the \emph{flat coordinates}
$t^\mu$. Greek indices will stand for flat coordinates with an
exception stated in Section~\ref{sec:cohft:two}.

A Frobenius manifold is called \emph{conformal} if it admits an
\emph{Euler vector field}, i.e. a vector field $E$ of the form
\begin{equation*}
  E = \sum_\mu (\alpha_\mu t^\mu + \beta_\mu) \frac{\partial}{\partial t^\mu},
\end{equation*}
such that the quantum product, the unit and the metric are
eigenfunctions of the Lie derivative $L_E$ with eigenvalues $1$, $-1$
and $2 - \delta$ respectively. Here $\delta$ is a rational number
called conformal dimension. Assuming that $A$ itself is the ring of
(formal) functions of a variety $X$ we say that the Frobenius manifold
is \emph{quasi-conformal} if there is vector field $E$ on $X \times V$
satisfying the axioms of an Euler vector field.

A CohFT $\Omega_{g, n}$ is called \emph{homogeneous}
(\emph{quasi-homogeneous}) if its Frobenius manifold is conformal
(quasi-conformal) and the extended CohFT is an eigenvector of of $L_E$
of eigenvalue $(g - 1)\delta + n$. As the name suggests CohFTs are
homogeneous if they carry a grading such that all natural structures
are homogeneous with respect to the grading.

We say that the Frobenius manifold $V$ is \emph{semisimple} if there
is a basis of idempotent vector fields $\epsilon_i$ defined after
possible base extension of $A$. The idempotents can be formally
integrated to \emph{canonical coordinates} $u_i$. We will use roman
indices for them. Let $\mathbf u$ be the diagonal matrix with entries
$u_i$ and $\Psi$ be the transition matrix from the basis of normalized
idempotents corresponding to the $u_i$ to the flat basis $e_i$.

The $R$-matrix from the trivial theory to $\Omega^p$ satisfies a
differential equation which is related to the \emph{quantum
  differential equation}
\begin{equation*}
  z \frac{\partial}{\partial t^\alpha} S_j = e_\alpha \star S_j
\end{equation*}
for vectors $S_j$. We assemble the $S_j$ into a matrix $S$.
\begin{prop}[see \cite{Gi01a}]
  \label{prop:fund}
  If $V$ is semisimple and after a choice of canonical coordinates
  $u_i$ has been made, there exists a fundamental solution $S$ to the
  quantum differential equation of the form
  \begin{equation} \label{eq:SR}
    S = \Psi R e^{\mathbf u/z},
  \end{equation}
  such that $R$ satisfies the symplectic condition $R(z)R^t(-z) =
  1$. The matrix $R$ is unique up to right multiplication by a
  diagonal matrix of the form
  \begin{equation*}
    \exp(a_1 z + a_3 z^3 + a_5 z^5 + \dotsb)
  \end{equation*}
  for constant diagonal matrices $a_i$.

  In the case that there exists an Euler vector field $E$, there is a
  unique matrix $R$ defined from a fundamental solution $S$ by
  \eqref{eq:SR} which satisfies the homogeneity
  \begin{equation*}
    z \frac{\mathrm d}{\mathrm dz} R + L_E R = 0.
  \end{equation*}
  Such an $R$ automatically satisfies the symplectic condition.
\end{prop}
\begin{rem}
  The matrix $R$ should be thought as the matrix representation of an
  endomorphism in the basis of normalized idempotents. The symplectic
  condition in Proposition~\ref{prop:fund} is then the same as in
  Definition~\ref{def:slg}.
\end{rem}
\begin{rem}
  The exponential in \eqref{eq:SR} has to be thought as a formal
  expression. All the quantities in Proposition~\ref{prop:fund} are
  only defined after base change of $A$ necessary to define the canonical
  coordinates.
\end{rem}
\begin{rem}
  The quantum differential equation is equivalent to the differential
  equation
  \begin{equation} \label{eq:qdercan}
    [R, \mathrm d\mathbf u] + z\Psi^{-1} \mathrm d(\Psi R) = 0
  \end{equation}
  for $R$.
\end{rem}

In the conformal case Teleman showed that the uniquely determined
homogeneous $R$-matrix of Proposition~\ref{prop:fund} is the one
appearing in the reconstruction, taking the trivial theory to the given
one.

Equivariant projective spaces $\PP^m$ only give a quasi-conformal
Frobenius manifold. However Givental showed, and we will recall the
proof in Appendix~\ref{sec:givenloc}, that in this case in the
reconstruction one should take $R$ such that in the classical limit $q
\to 0$ it assumes the diagonal form
\begin{equation} \label{eq:limit}
  R|_{q = 0} = \exp(\diag(b_0, \dotsc, b_m)),
\end{equation}
where, using the notation from Section~\ref{sec:cohft:two},
\begin{equation*}
  b_j = \sum_{i = 1}^\infty \frac{B_{2i}}{2i(2i - 1)} \sum_{l \neq j} \left(\frac z{\lambda_l - \lambda_j}\right)^{2i - 1}.
\end{equation*}
The $R$-matrix is uniquely determined by this additional property and
the homogeneity.

\subsection{The two CohFTs} \label{sec:cohft:two}

The cohomological field theory corresponding to the
$A_{m + 1}$-singularity $f(X) = X^{m + 2}/(m + 2)$ is defined using
Witten's $(m + 2)$-spin class on the moduli of curves with
$(m + 2)$-spin structures. See \cite{PPZ15} for a discussion of
different constructions of Witten's class. In comparison to
\cite{PPZ15} we use a different normalization for Witten's class and a
different basis for the free module in order to have a more direct
comparison to the $\PP^m$-theory.

The CohFT is based on the rank $(m + 1)$ free module of versal
deformations
\begin{equation*}
  f_t(X) = \frac{X^{m + 2}}{m + 2} + t^m X^m + \dotsb + t^1 X + t^0
\end{equation*}
of $f$. In this article, using the deformation from
Section~\ref{sec:cohft:frobenius}, we will view the CohFT as being
based on
\begin{equation*}
  k_{A_{m + 1}} = \QQ[t^1, \dotsc, t^m],
\end{equation*}
the space of regular functions on the Frobenius manifold where the
$t^0$-coordinate vanishes. Because of dimension constraints we do not
need to look at formal functions, and because the CohFT stays constant
along the $t^0$ direction we can restrict to the $(t^0 = 0)$-subspace.

The algebra structure is given by $k_{A_{m + 1}}[X] / (f'_t)$, where
$X^\mu$ corresponds to $\frac\partial{\partial t^\mu}$. The metric is
given by the residue pairing
\begin{equation*}
  \eta(a, b) = \frac 1{2\pi\imath} \oint \frac{ab}{f'_t(X)} \mathrm dX.
\end{equation*}
Written as a matrix in the basis $1, \dotsc, X^m$, the metric $\eta$
has therefore zeros above the antidiagonal, ones at the antidiagonal
and again zeros in the first antidiagonal below it. Notice also that
$\eta$ has no dependence on $t^1$. Therefore, while the $t^\mu$ do not
give a basis of flat vector fields on the Frobenius manifold, there is
a triangular matrix independent of $t^1$, sending the $1, \dotsc, X^m$
to a basis of flat vector fields such that $X$ is mapped to
itself. With this we can pretend that the $t^\mu$ were flat
coordinates if we consider in the quantum differential equation only
differentiation by $t^1$.

\medskip

For $(\CC^*)^{m + 1}$-equivariant $\PP^m$ the CohFT is based on the
equivariant Chow ring
\begin{equation*}
  A_{(\CC^*)^{m + 1}}^* (\PP^m)[\![q]\!] \cong k_{\PP^m}[H] / \prod_{i = 0}^m (H - \lambda_i),
\end{equation*}
of $\PP^m$, an $(m + 1)$-dimensional free module over
\begin{equation*}
  k_{\PP^m} = \QQ[\lambda_0, \dotsc, \lambda_m][\![q]\!],
\end{equation*}
and depends on the Novikov variable $q$ and the torus parameters
$\lambda_i$. We will not consider the deformation from
Section~\ref{sec:cohft:frobenius}. The algebra structure is given by
the small quantum equivariant Chow ring
\begin{equation*}
  QA_{(\CC^*)^{m + 1}}^* (\PP^m) \cong k_{\PP^m}[H] / \left(\prod_{i = 0}^m (H - \lambda_i) - q\right)
\end{equation*}
and the pairing is the Poincaré pairing
\begin{equation*}
  \eta(a, b) = \frac 1{2\pi\imath} \oint \frac{ab}{\prod_{i = 0}^m (H - \lambda_i)} \mathrm dH
\end{equation*}
in the equivariant Chow ring.

\medskip

To match up this data we set
\begin{align*}
  X
  =& H - \bar\lambda, \\
  X^{m + 1} + \sum_{\mu = 0}^{m - 1} (\mu + 1)t^{\mu + 1} X^\mu
  =& \prod_{i = 0}^m (X + \bar\lambda - \lambda_i) - q,
\end{align*}
where
\begin{equation*}
  \bar\lambda = \sum_{i = 0}^m \frac{\lambda_i}{m + 1}.
\end{equation*}
So in particular
\begin{equation*}
  t^1 = -q + \prod_{i = 0}^m (\bar\lambda - \lambda_i) =: -q - \lambda
\end{equation*}
and we have described a ring map
\begin{equation*}
  \Phi: k_{A_{m + 1}}[\lambda] \to k_{\PP^m},
\end{equation*}
whose image are the polynomials, symmetric in the torus parameters and
vanishing if all torus parameters coincide. Therefore, after base
extension, the Frobenius algebras from the $A_{m + 1}$-singularity and
equivariant $\PP^m$ match completely up.

On the $\PP^m$-side, let $Q_i$ be the power series solution to
\begin{equation*}
  \prod_{i = 0}^m (Y + \bar\lambda - \lambda_i) = q
\end{equation*}
with limit $\lambda_i - \bar\lambda$ as $q \to 0$. In particular, the
$Q_i$ are solutions to
\begin{equation*}
  Y^{m + 1} + \sum_{\mu = 0}^{m - 1} (\mu + 1)t^{\mu + 1} Y^\mu.
\end{equation*}
On the $A_{m + 1}$-side, let the $Q_i$ be the solutions to this
equation in any order. On both sides we can then define
\begin{equation*}
  \Delta_i = \prod_{j \neq i} (Q_i - Q_j)
  = (m + 1) Q_i^m - \sum_{\mu = 1}^{m - 1} (\mu + 1)\mu t^{\mu + 1} Q_i^{\mu - 1}
\end{equation*}
and the \emph{discriminant}
\begin{equation*}
  \disc = \prod_i \Delta_i \in k_{A_{m + 1}}.
\end{equation*}
The choice of the $Q_i$ gives a bijection between the idempotents
\begin{equation*}
  \epsilon_i = \frac{\prod_{j \neq i} (X - Q_j)}{\Delta_i}.
\end{equation*}
We will also need to make a choice of square roots of the $\Delta_i$
to be able to define the normalized idempotents
\begin{equation*}
  \tilde\epsilon_i = \frac{\prod_{j \neq i} (X - Q_j)}{\sqrt{\Delta_i}}.
\end{equation*}

The $A_{m + 1}$-theory is conformal with Euler vector field
\begin{equation*}
  E = \sum_{i = 1}^m \frac{m + 2 - i}{m + 2} t^\mu \frac{\partial}{\partial t^\mu},
\end{equation*}
while the equivariant $\PP^m$-theory is semi-conformal with Euler
vector field
\begin{equation*}
  E = (m + 1)q \frac{\partial}{\partial q} + \sum_{i = 0}^m \lambda_i \frac{\partial}{\partial \lambda_i}.
\end{equation*}

\subsection{Relations from CohFTs} \label{sec:cohft:rels}

Let $\Omega$ be a semisimple CohFT defined on $V$ over $A$. Formal
properties of the reconstruction theorem will imply tautological
relations. The main point is that the $R$-matrix from the trivial
theory written in flat coordinates lives only in
\begin{equation*}
  \End(V \otimes_A B)[\![z]\!],
\end{equation*}
for some $\QQ$-algebra extension $B$\footnote{In our examples
  $B = A[\disc^{-1}]$.} of $A$. Let $C$ be the $A$-module quotient
fitting into the exact sequence
\begin{equation} \label{eq:ABC}
  0 \to A \to B \xrightarrow{p} C \to 0.
\end{equation}
The reconstruction gives elements
\begin{equation*}
  \overline\Omega_{g, n} \in \strata_{g, n} \otimes (V^*)^{\otimes n} \otimes B.
\end{equation*}
However since we have started out with a CohFT defined over $A$, we
know that the projection of
\begin{equation*}
  p(\overline\Omega_{g, n}) \in \strata_{g, n} \otimes (V^*)^{\otimes n} \otimes C
\end{equation*}
to $R^*(\Mbar_{g, n}) \otimes (V^*)^{\otimes n} \otimes C$ has to
vanish\footnote{Assuming that reconstruction holds in this
  case.}. Since $C$ is a $\QQ$-vector space, we obtain a system of
vector spaces $T_{g, n}^\Omega$ of relations. The complete system
$\bar T_{g, n}^\Omega$ of \emph{tautological relations obtained from
  the CohFT $\Omega$} is the vector space generated by
\begin{equation*}
  \xi_* (\pi_*(T_{g_1, n_1 + m}^\Omega P) \times \strata_{g_2, n_2} \times \dotsb \times \strata_{g_k, n_k}),
\end{equation*}
where $P$ is the vector space of polynomials in $\psi$-classes, and
$\xi_*$ and $\pi_*$ are the formal analogues of the push-forwards
along gluing and forgetful maps.

We say that a vector space of tautological relations $T_{g, n}$
implies another $T'_{g, n}$ if the vector space, obtained from $T_{g,
  n}$ by the completion process as described right above, is contained
in $T'_{g, n}$. Using this definition we can also define an
equivalence relation between vector spaces of tautological relations.

Let us describe two relation preserving actions on the space of all
CohFTs on $V$ over $A$. The first is an action of the multiplicative
monoid of $A$. The action of $\varphi \in A$ is given by
multiplication by $\varphi^d$ in codimension $d$. This replaces the
$R$-matrix $R(z)$ of the theory by $R(\varphi z)$. Since
multiplication by $\varphi$ is well-defined in $C$, relations are
preserved. The second action is the action of an $R$-matrix defined
over $A$.

The second action automatically proves equivalence of relations since
$R$-matrices are always invertible. Similarly, the first action proves
equivalence if $\varphi$ is invertible.

\emph{Extending scalars} also preserves relations. By this we mean
tensoring $\Omega$ with $A \to A'$ under the condition that this
preserves the exactness of the sequence \eqref{eq:ABC}. We call the
special case when $A' = A/I$ for some ideal $I$ of $A$ a \emph{limit}.
If $C \to C \otimes_A A'$ is injective, extending scalars proves an
equivalence of relations.

Let us again state our now well-defined results.
\setcounter{thm}{0}
\begin{thm} \label{thm:limit}
  The relations from the equivariant Gromov-Witten theory of $\PP^m$
  imply the $(m + 2)$-spin relations, both CohFTs as defined in
  Section~\ref{sec:cohft:two}.
\end{thm}
The main statement necessary to be proven here is that the $R$-matrix
for $\PP^m$ after replacing $z \mapsto z\lambda^{-1}$ admits the limit
$\lambda^{-1} \to 0$ and that this limit is the $R$-matrix for the
$A_{m + 1}$-theory. In order for this to make sense, one uses the
matchup from Section~\ref{sec:cohft:two} and views both as being
defined over
\begin{equation*}
  \QQ[\![\lambda_0, \dotsc, \lambda_m, q]\!][\lambda^{-1}]
\end{equation*}
In Section~\ref{sec:osc} we will see that for both original theories
to define the $R$-matrix it is enough to localize by $\disc$. So the
extension of scalars does not lose relations.

Motivated from Section~\ref{sec:osc:spin} let us call the limit $t^2,
\dotsc, t^m = 0$ the \emph{Airy limit}. For $\PP^m$ the Airy limit
concretely means, assuming the sum of all torus weights is zero, that
we restrict ourselves to the case that up to a factor the torus
weights are the $(m + 1)$-th roots of unity.
\begin{thm} \label{thm:equiv}
  In the Airy limit the $\PP^m$- and $(m + 2)$-spin relations are
  equivalent.
\end{thm}
The main point for the proof is to show there is a series
\begin{equation*}
  \varphi \in \lambda \QQ[\![t^1\lambda^{-1}]\!],
\end{equation*}
and an $R$-matrix $R$ without poles in $\disc$ such that the Airy
limit $\PP^m$-$R$-matrix is obtained from the Airy limit $A_{m +
  1}$-$R$-matrix by applying the transformation $z \mapsto z\varphi$,
followed by the action of $R$. We will show in the proof that there is
only one possible choice for $\varphi$. For Theorem~\ref{thm:equiv}
both theories can be viewed as living over
\begin{equation*}
  \QQ[\![\lambda_0, \dotsc, \lambda_m, q, t^1\lambda^{-1}]\!][\lambda^{-1}]/(t^1 + q + \lambda, t^2, \dotsc, t^m).
\end{equation*}

In Section~\ref{sec:noequiv} we will give evidence that the method of
proof of Theorem~\ref{thm:equiv} does not work outside the Airy
limit. What we will show is that assuming a procedure as in the proof
of Theorem~\ref{thm:equiv} exists and is well-defined in the Airy
limit, the information that $\varphi$ was unique in the limit implies
that the $R$-matrix in the $R$-matrix action cannot be defined over
the base ring.

\subsubsection*{Relations from degree vanishing}

The more classical way of \cite{PPZ15} and \cite{rspin} to obtain
tautological relations works by considering cohomological degrees:
Assume that $\Omega$ is in addition quasi-homogenous for an Euler
vector field $E$ and that all $\beta_i$ vanish and all $\alpha_i$ are
positive. Then the quasi-homogeneity implies that the cohomological
degree of $\Omega_{g, n}(\frac\partial{\partial t^{i_1}}, \dotsc,
\frac\partial{\partial t^{i_n}})$ is bounded by
\begin{equation*}
  (g - 1)\delta + n - \sum_j \alpha_{i_j}.
\end{equation*}
However the reconstructed theory might also contain terms of higher
cohomological degree. These thus have to vanish, giving tautological
relations.

Notice that these relations coming from degree considerations are
implied from the relations we have described previously: With respect
to the grading on $B$ induced by the Euler vector field, no element of
$A$ has negative degree. Therefore the negative degree parts of $B$
and $C$ are isomorphic. Thus, the homogeneity of the CohFT implies
that the degree vanishing relations are obtained from the previous
relations by restricting to the negative degree part of $C$.

The way of obtaining tautological relations by looking at poles in the
discriminant has already previously been studied by D. Zvonkine.

\section{Oscillating integrals} \label{sec:osc}

\subsection{For the \texorpdfstring{$A_{m + 1}$}{A(m +
    1)}-singularity} \label{sec:osc:spin}

We want to describe the $A_{m + 1}$-$R$-matrix in terms of asymptotics
of oscillating integrals. For the purposes of this article the
integrals can be treated as purely formal objects.

The quantum differential equation with one index lowered says that
\begin{align*}
  z\frac{\partial}{\partial t^1} S_{\mu k}
  =& S_{(\mu + 1)k}, && \text{for }\mu < m,\\
  z\frac{\partial}{\partial t^1} S_{mk}
  =& -\sum_{\mu = 0}^{m - 1} (\mu + 1)t^{\mu + 1} S_{\mu k},&&
\end{align*}
where the Greek indices stand for components in the basis of the
$X^\mu$. It is not difficult to see that the oscillating integrals
\begin{equation*}
  \frac 1{\sqrt{-2\pi z}} \int_{\Gamma_k} x^\mu \exp(f_t(x)/z) \mathrm dx,
\end{equation*}
where $f_t$ as before is the deformed singularity, for varying cycles
$\Gamma_k$, if convergent, provide solutions to this system of
differential equations, and also satisfy homogeneity with respect to
the Euler vector field.

For generic choices of parameters, to each critical point $Q_k$ there
corresponds a cycle $\Gamma_k$ constructed via the Morse theory of
$\Re(f_t(x)/z)$, which moves through that critical point in the
direction of steepest descent and avoids all other critical points. By
moving to the critical point and scaling coordinates we obtain
\begin{multline*}
  S_{\mu k}
  = \frac{e^{u_k/z}}{\sqrt{2\pi\Delta_k}} \int \left(\frac{x(-z)^{1/2}}{\sqrt{\Delta_k}} + Q_k\right)^\mu \\
  \exp\left(-\sum_{l = 2}^{m + 2} \frac{x^l(-z)^{(l - 2)/2}}{l!} \frac{f_t^{(l)}(Q_k)}{\Delta_k^{l/2}} \right) \mathrm dx,
\end{multline*}
where $u_k = f_t(Q_k)$. By the method of steepest descend, we obtain
the asymptotics as $z \to 0$ by expanding the integrand as a formal
power series in $z$ and integrating from $-\infty$ to $\infty$. Since
the $(l = 2)$-term in the sum is $-x^2/2$, we can use the formula for
the moments of the Gaussian distribution to write the asymptotics of
$\sqrt{\Delta_k} e^{-u_k/z} S_{\mu k}$ as a formal power series in $z$
with values in $k_{A_{m + 1}} [Q_k, \Delta_k^{-1}]$.

The entries of the $R$-matrix are then given by
\begin{equation*}
  R_{ik}
  \asymp \frac 1{\sqrt{\Delta_i}} e^{-u_k/z} \prod_{j \neq i} \left(\frac\partial{\partial t^1} - Q_j\right) S_{0k}.
\end{equation*}
Noticing that the change of basis from normalized idempotents to the
basis $1, X, \dotsc, X^m$ can be defined over $k_{A_{m + 1}} [Q_k,
\Delta_k^{-1/2}]$, recalling that $\disc = \prod \Delta_i$ and
applying Galois theory, we see that the endomorphism $R$ is defined
over $k_{A_{m + 1}}[\disc^{-1}]$.

\medskip

In the Airy limit $t^2, \dotsc, t^m \to 0$ the quantum differential
equation becomes the slightly modified higher Airy differential
equation \cite{Ko92}
\begin{equation*}
  \left(z\frac{\partial}{\partial t^1}\right)^{m + 1} S_{0k} = -t^1 S_{0k}.
\end{equation*}
The entries of the $R$-matrix in this case are therefore related to
the asymptotic expansions of the higher Airy functions and their
derivatives when their complex argument approaches $\infty$.

In the case of the $A_2$-singularity we do not need to take any limit
and discover the hypergeometric series $A$ and $B$ of Faber-Zagier in
the expansions of the (slightly modified) usual Airy function
\begin{multline*}
  \frac{e^{\frac 23 (t^1)^{3/2}/z}}{\sqrt{-2\pi z}} \int_{\Gamma_k} e^{\left(\frac{x^3}3 + t^1x\right)/z} \mathrm dx
  \asymp \frac 1{\sqrt{2\pi\Delta}} \int\limits_{-\infty}^\infty \exp\left(-\frac{x^2}2 - \frac{x^3}3 \frac{\sqrt{-z}}{\Delta^{3/2}} \right) \mathrm dx \\
  \asymp \frac 1{\sqrt{\Delta}} \sum_{i = 0}^\infty \frac{(6i - 1)!!}{(2i)!} \left(\frac{-z}{9\Delta^3}\right)^i
  = \frac 1{\sqrt{\Delta}} \sum_{i = 0}^\infty \frac{(6i)!}{(3i)!(2i)!} \left(\frac{-z}{72\Delta^3}\right)^i
\end{multline*}
and a derivative of it
\begin{equation*}
  \frac{e^{\frac 23 (t^1)^{3/2}/z}}{\sqrt{-2\pi z}} \int_{\Gamma_k} x e^{\left(\frac{x^3}3 + t^1x\right)/z} \mathrm dx
  \asymp \frac{\sqrt{-t^1}}{\sqrt{\Delta}} \sum_{i = 0}^\infty \frac{(6i)!}{(3i)!(2i)!} \frac{1 + 6i}{1 - 6i} \left(\frac{-z}{72\Delta^3}\right)^i.
\end{equation*}
Here $\Delta = 2\sqrt{-t^1}$. The cycle $\Gamma_k$ determines which
square-root of $(-t^1)$ we take.

\subsection{For equivariant \texorpdfstring{$\PP^m$}{Pm}} \label{sec:osc:PP}

Givental \cite{Gi01b} has given explicit solutions to the quantum
differential equation for projective spaces in the form of complex
oscillating integrals. Let us recall their definition and see how they
behave in the match up with the $(m + 2)$-spin theory.

Using the divisor axiom of Gromov-Witten invariants, the quantum
differential equation implies the differential equations
\begin{equation*}
  (D + \lambda_i) S_i = H \star S_i.
\end{equation*}
for the fundamental solutions $S_i$ at the origin. Here we have
written $D = z q\frac{\partial}{\partial q}$. Equivalently, the
equation says
\begin{equation*}
  D (S_i e^{\ln(q) \lambda_i/z})
  = H \star S_i e^{\ln(q) \lambda_i/z}.
\end{equation*}
Therefore the entries of $S$ with one index lowered satisfy
\begin{equation*}
  (D - \bar\lambda)(S_{\mu i} e^{\ln(q) \lambda_i/z})
  = S_{(\mu + 1)i} e^{\ln(q) \lambda_i/z},
\end{equation*}
where $S_{(m + 1)i}$ is defined such that
\begin{equation*}
  \prod_{j = 0}^m (D - \lambda_j) (S_{0i} e^{\ln(q) \lambda_i/z})
  = qS_{0i} e^{\ln(q) \lambda_i/z}.
\end{equation*}
The Greek indices stand for the basis of flat vector fields
corresponding to $1, H - \bar\lambda, \dotsc, (H - \bar\lambda)^m$.

Givental's oscillating integral solutions for $S_{0i}$ are stationary
phase expansions of the integrals
\begin{equation*}
  S_{0i} e^{\ln(q) \lambda_i/z}
  = (-2\pi z)^{-m/2} \int\limits_{\Gamma_i \subset \{\sum T_j = \ln q\}} e^{F_i(T)/z} \omega
\end{equation*}
along $m$-cycles $\Gamma_i$ through a specific critical point of
$F_i(T)$ inside a $m$-dimensional $\CC$-subspace of $\CC^{m + 1}$,
where
\begin{equation*}
  F_i(T) = \sum_{j = 0}^m (e^{T_j} + \lambda_jT_j).
\end{equation*}
The form $\omega$ is the restriction of $\mathrm dT_0 \wedge \dotsb
\wedge \mathrm dT_m$. To see that the integrals are actual solutions,
notice that applying $D - \lambda_j$ to the integral has the same
effect as multiplying the integrand by $e^{T_j}$.

There are $m + 1$ possible critical points at which one can do a
stationary phase expansion of $S_{0i}$. Let us write
$P_i = Q_i + \bar\lambda$ for the solution to
\begin{equation*}
  \prod_{i = 0}^m (X - \lambda_i) = q
\end{equation*}
with limit $\lambda_i$ as $q \to 0$. For each $i$ we need to choose
the critical point $e^{T_j} = P_i - \lambda_j$ in order for the factor
\begin{equation*}
  \exp(u_i / z) := \exp\left(\left(\sum_{j = 0}^m (P_i - \lambda_j + \lambda_j \ln(P_i - \lambda_j)) - \lambda_i\ln(q)\right)/z\right)
\end{equation*}
of $S_{0i}$ to be well-defined in the limit $q \to 0$. Shifting the
integral to the critical point and scaling coordinates by $\sqrt{-z}$
we find
\begin{equation*}
  S_{0i} = e^{u_i/z} \int \exp\left(-\sum_j (Q_i - \bar\lambda_j) \sum_{k = 3}^\infty \frac{T_j^k (-z)^{(k - 2)/2}}{k!}\right) \mathrm d\mu_i
\end{equation*}
for the conditional Gaussian distribution
\begin{equation*}
  \mathrm d\mu_i = (2\pi)^{-m/2} \exp\left(-\sum_j (Q_i - \bar\lambda_j) \frac{T_j^2}2\right) \omega.
\end{equation*}
The covariance matrices are given by
\begin{equation*}
  \sigma_i(T_k, T_l)
  = \frac 1{\Delta_i}
  \begin{cases}
    -\prod_{j \notin \{k, l\}} (Q_i - \bar\lambda_j), & \text{for }k \neq l, \\
    \sum_{m \neq k} \prod_{j \notin \{k, m\}} (Q_i - \bar\lambda_j), & \text{for }k = l.
  \end{cases}
\end{equation*}
From here we can see that the integral is symmetric in the
$\bar\lambda_j$ and therefore we can write its asymptotics as
$z \to 0$ completely in terms of data from $A_{m + 1}$. Since odd
moments of Gaussian distributions vanish we find that
$e^{-u_i/z} S_{0i}$ is a power series in $z$ with values in
$\Delta_i^{-1/2}k_{A_{m + 1}}[Q_i, \Delta_i^{-1}, \lambda]$.

So the entries of the $R$-matrix in the basis of normalized
idempotents are given by
\begin{equation*}
  \Delta_k^{-1/2} \prod_{j \neq k}(D + \lambda_i - P_j) \ (e^{-u_i/z}S_{0i}).
\end{equation*}
Since $\frac{\mathrm dP_i}{\mathrm dq} = \frac 1{\Delta_i}$ these
entries are in
\begin{equation*}
  k_{A_{m + 1}}[Q_0, \dotsc, Q_m, \Delta_0^{-1/2}, \dotsc, \Delta_m^{-1/2}, \lambda].
\end{equation*}
So, with the arguments from Section~\ref{sec:osc:spin}, the
endomorphism $R$ can be defined over $k_{\PP^m}[\disc^{-1}]$.

We need to check that the $R$-matrix given in terms of oscillating
integrals behaves correctly in the limit $q \to 0$. By definition, in
this limit $P_i \to \lambda_i$. By symmetry it is enough to consider
the $0$-th column. Set $x_i = e^{T_i}$. Then
\begin{multline*}
  \lim\limits_{q \to 0} R_{j0}
  \asymp \lim\limits_{q \to 0} e^{-u_0/z} \Delta_0^{-1/2} \prod_{k \neq j}(zq\frac{\mathrm d}{\mathrm dq} + \lambda_j - \lambda_k) \ S_{00} \\
  = \lim\limits_{q \to 0} \frac{e^{-u_0/z}}{\sqrt{\Delta_0} (-2\pi z)^{m/2}} \int e^{(\sum_k (e^{T_k} - (\lambda_0 - \lambda_k) T_k)) / z + \sum_{k \neq j} T_k} \omega \\
  = \lim\limits_{q \to 0} \frac{e^{-u_0/z}}{\sqrt{\Delta_0}(-2\pi z)^{m/2}} \int e^{(\sum\limits_{k \neq 0} (x_k - (\lambda_0 - \lambda_k) T_k) + \frac q{\prod\limits_{k \neq 0} x_k}) / z} \prod_{k \neq j} x_j \bigwedge_{k = 1}^m \mathrm dT_k.
\end{multline*}
In the last step we have moved to the chart
\begin{equation*}
  x_0 = \frac q{\prod_{j \neq 0} x_j}.
\end{equation*}
Since in this chart $\lim_{q \to 0} x_0 = 0$, we have that $R_{j0}$
vanishes unless $j = 0$. On the other hand in the limit $q \to 0$ the
integral for $R_{00}$ splits into one-dimensional integrals
\begin{equation*}
  \lim\limits_{q \to 0} R_{00}
  \asymp \lim\limits_{q \to 0} \frac{e^{-u_0/z}}{\sqrt{\Delta_0}(-2\pi z)^{m/2}} \prod_{k \neq 0} \int\limits_0^\infty e^{(x - (\lambda_0 - \lambda_k) \ln(x)) / z} \mathrm dx.
\end{equation*}
Let us temporarily set $z_k = -z/(\lambda_0 - \lambda_k)$. The
prefactors also split into pieces in the limit and we calculate the
factor corresponding to $k$ to be
\begin{multline*}
  \hspace{-1em}\frac{e^{(1 - \ln(\lambda_0 - \lambda_k)) / z_k}}{\sqrt{-2\pi z(\lambda_0 - \lambda_k)}} \int\limits_0^\infty e^{(x - (\lambda_0 - \lambda_k) \ln(x))/z} \mathrm dx
  = \frac{e^{(1 - \ln(1 / z_k)) / z_k}}{\sqrt{2\pi / z_k }} \Gamma\left(1 + \frac 1{z_k}\right)\\
  = \frac{e^{(1 - \ln(1 / z_k)) / z_k}}{\sqrt{2\pi z_k}} \Gamma\left(\frac 1{z_k}\right)
  \asymp \exp\left(\sum_{l = 1}^\infty \frac{B_{2l}}{2l(2l - 1)} \left(\frac z{\lambda_k - \lambda_0}\right)^{2l - 1} \right),
\end{multline*}
using Stirling's approximation of the gamma function in the last
step. So the product of the factors gives the expected limit
\eqref{eq:limit} of $R_{00}$ for $q \to 0$. This calculation gives a
proof for the results \cite{Io05} of Ionel on the main generating
function used in \cite{PaPi13P} and \cite{Ja13P} without having to
use Harer's stability results.

\section{\texorpdfstring{$\PP^m$}{Pm} relations imply \texorpdfstring{$(m + 2)$}{(m + 2)}-spin
  relations} \label{sec:limit}

We prove Theorem~\ref{thm:limit} in this section. As already mentioned,
for this it is enough to show that, after the change $z \mapsto
z\lambda^{-1}$, the $\PP^m$-$R$-matrix converges to the $A_{m +
  1}$-$R$-matrix in the limit $\lambda \to \infty$. For this we have
to compare the differential equations satisfied by the $R$-matrices.

Inserting the vector field corresponding to the hyperplane into
\eqref{eq:qdercan} and using the divisor equation as in
Section~\ref{sec:osc:PP} gives the equation
\begin{equation} \label{eq:qdeRP}
  [R_{\PP^m}, \xi] + zq \frac{\mathrm dR_{\PP^m}}{\mathrm dq} + zq\Psi^{-1} \frac{\mathrm d\Psi}{\mathrm dq} R_{\PP^m} = 0,
\end{equation}
where $\xi$ denotes the diagonal matrix of quantum multiplication by
$H - \bar\lambda$.

\begin{lem} \label{lem:limit}
  $R_{\PP^m}(z/\lambda)$ admits a limit $R$ for $\lambda \to \infty$. The
  matrix $R$ satisfies
  \begin{align*}
    [R, \xi] + z \frac{\mathrm dR}{\mathrm dt^1} + z\Psi^{-1} \frac{\mathrm d\Psi}{\mathrm dt^1} R = 0 \\
    z \frac{\mathrm dR}{\mathrm dz} + \sum_{\mu = 1}^m \frac{m + 2 - \mu}{m + 2} t^\mu \frac{\mathrm dR}{\mathrm dt^\mu} = 0
  \end{align*}
\end{lem}
\begin{proof}
  The $\PP^m$-$R$-matrix satisfies the homogeneity property
  \begin{equation*}
    z \frac{\mathrm dR_{\PP^m}}{\mathrm dz} + (m + 1) q \frac{\mathrm dR_{\PP^m}}{\mathrm dq} + \sum_{i = 0}^m \lambda_i \frac{\mathrm dR_{\PP^m}}{\mathrm d\lambda_i} = 0.
  \end{equation*}
  So $R'(z) := R_{\PP^m}(z/\lambda)$ written with the $A_{m +
    1}$-variables satisfies
  \begin{equation*}
    (m + 2) z \frac{\mathrm dR'}{\mathrm dz} + (m + 1) \lambda \frac{\mathrm dR'}{\mathrm d\lambda} + \sum_{\mu = 1}^m (m + 2 - \mu)t^\mu \frac{\mathrm dR'}{\mathrm dt^\mu} = 0.
  \end{equation*}

  The main differential equation satisfied by $R'$ is
  \begin{equation*}
    [R', \xi] + z\left(1 + \frac{t^1}\lambda\right) \frac{\mathrm dR'}{\mathrm dt^1}
    + z\left(1 + \frac{t^1}\lambda\right)\Psi^{-1} \frac{\mathrm d\Psi}{\mathrm dt^1} R' = 0.
  \end{equation*}
  From the expression of $R_{\PP^m}$ in terms of oscillating integrals
  we know that the entries of the $z^i$-part $R'_i$ of $R'$ live in
  \begin{equation*}
    \lambda^{-i} k_{A_{m + 1}}[Q_0, \dotsc, Q_m, \Delta_0^{-1/2}, \dotsc, \Delta_m^{-1/2}, \lambda].
  \end{equation*}
  To show that the limit exists we need to show that $\lambda$ occurs
  in no positive power. We will show this by induction by $i$. It
  certainly holds for $R'_0 = 1$. Since $\xi$ is diagonal with
  pairwise distinct entries $Q_j$, the $z^i$-part of the differential
  equation determines the off-diagonal coefficients of $R'_i$ in terms
  of $R'_{i - 1}$. Because $\Psi^{-1} \frac{\mathrm d\Psi}{\mathrm
    dt^1}$ does not depend on $\lambda$, the off-diagonal coefficients
  of $R'_i$ will admit the limit $\lambda \to \infty$. Since
  $\Psi^{-1} \frac{\mathrm d\Psi}{\mathrm dt^1}$ in general vanishes
  on the diagonal the diagonal coefficient of the $z^{i + 1}$-part of
  the differential equation determines the diagonal of $\frac{\mathrm
    dR'_i}{\mathrm dt^1}$ from the off-diagonal entries of
  $R'_i$. Apart from a possible term constant in $t^1$ we therefore
  know that also the diagonal entries of $R'_i$ admit the limit.

  Let us consider such a possible ambiguity $a_i$. Since all products
  of $\Delta_j$ have dependence in $t^1$, the ``denominator'' of $a_i$
  can only be a power of $\lambda$ less than $i$. However then $a_i$
  cannot possibly satisfy the homogeneity. By induction therefore the
  limit $R$ exists. The properties of $R$ easily follow from the
  corresponding ones of $R'$.
\end{proof}

By inserting the vector field $\frac\partial{\partial t^1}$ into
\eqref{eq:qdercan} and similar arguments as in the proof of
Lemma~\ref{lem:limit} one can show the following lemma.
\begin{lem}
  The $A_{m + 1}$-$R$-matrix is uniquely determined from the differential
  equation
  \begin{equation*}
    [R_{A_{m + 1}}, \xi] + z \frac{\mathrm dR_{A_{m + 1}}}{\mathrm dt^1} + z\Psi^{-1} \frac{\mathrm d\Psi}{\mathrm dt^1} R_{A_{m + 1}} = 0,
  \end{equation*}
  the homogeneity
  \begin{equation*}
    z \frac{\mathrm dR_{A_{m + 1}}}{\mathrm dz} + \sum_{\mu = 1}^m \frac{m + 2 - \mu}{m + 2} t^\mu \frac{\mathrm dR_{A_{m + 1}}}{\mathrm dt^\mu} = 0
  \end{equation*}
  and that the entries of the $z$-series coefficients of $R_{A_{m + 1}}$
  should lie in
  \begin{equation*}
    k_{A_{m + 1}}[Q_0, \dotsc, Q_m, \Delta_0^{-1/2}, \dotsc, \Delta_m^{-1/2}].
  \end{equation*}
\end{lem}

The lemmas imply that the modified $\PP^m$-$R$-matrix contains only
non-positive powers of $\lambda$ and the part constant in $\lambda$
equals the $A_{m + 1}$-$R$-matrix. Therefore the $A_{m + 1}$-relations
are contained in the modified $\PP^m$-relations as the
$\lambda^0$-part, and we have completed the proof of
Theorem~\ref{thm:limit}.

\section{Equivalence of relations} \label{sec:equiv}

We want to give a proof of Theorem~\ref{thm:equiv} in this section.
So we will consider the CohFTs in the Airy limit, i.e. with all
$t^\mu$ but $t := t^1$ set to zero. In this limit the metric becomes
$\eta(X^i, X^j) = \delta_{i + j, m}$, the quantum product stays
semisimple and the Euler vector field for the $A_m$-singularity
\begin{equation*}
  E = \frac{m + 1}{m + 2} t \frac\partial{\partial t}
\end{equation*}
is a multiple of $X$.

Rewriting \eqref{eq:qdeRP} for the $\PP^m$-$R$-matrix $\tilde
R_{\PP^m} = \Psi R_{\PP^m} \Psi^{-1}$ written in flat coordinates
gives
\begin{align*}
  [\tilde R_{\PP^m}, \xi] - zqL_E \tilde R_{\PP^m} + zq\tilde
  R_{\PP^m} \mu = 0,
\end{align*}
where $\xi$ is multiplication by $E$ in flat coordinates and $\mu =
-(L_E \Psi) \Psi^{-1}$.

We need to find a series $\varphi$ in $t$ and an $R$-matrix $R$
sending the modified $A_{m + 1}$-theory to equivariant $\PP^m$:
\begin{equation*}
  \tilde R_{\PP^m}(z) = R(z) \tilde R_{A_{m + 1}}(z\varphi).
\end{equation*}

We know that $\tilde R_{A_{m + 1}}$ satisfies
\begin{equation*}
  [\tilde R_{A_{m + 1}}, \xi] + zL_E \tilde R_{A_{m + 1}} - z\tilde R_{A_{m + 1}}\mu = 0
\end{equation*}
and the weighted homogeneity condition
\begin{equation*}
  \left(z\frac{\mathrm d}{\mathrm dz} + L_E\right) \tilde R_{A_{m + 1}} + [\mu, \tilde R_{A_{m + 1}}] = 0.
\end{equation*}
Putting these together we find that $R$ must satisfy
\begin{multline*}
  0 = [R, \xi] - zqL_ER + zq \frac{L_E \varphi}\varphi R\mu \\
  + \frac 1\varphi \left(q + \varphi - q \frac{L_E \varphi}\varphi\right)R [\tilde R_{A_{m + 1}}(z\varphi), \xi]\tilde R_{A_{m + 1}}^{-1}(z\varphi).
\end{multline*}
\begin{lem} \label{lem:notpol}
  The series $\tilde R_{A_{m + 1}} \xi \tilde R_{A_{m + 1}}^{-1}$ is
  not a polynomial in $z$.
\end{lem}
Because of the lemma and the homogeneity of $\tilde R_{A_{m + 1}}$ we
see that in order for $R$ to exist in the limit $\disc \to 0$ the
function $\varphi$ has to satisfy
\begin{equation*}
  q + \varphi - q \frac{L_E \varphi}\varphi = 0
\end{equation*}
or equivalently
\begin{equation*}
  -q^{-1} = \varphi^{-1} + L_E\varphi^{-1}.
\end{equation*}
There is a unique solution $\varphi^{-1}$ to this differential
equation. Concretely, we have
\begin{equation*}
  \varphi^{-1}
  = \lambda^{-1} \sum_{i = 0}^\infty \frac{m + 2}{m + 2 + i(m + 1)} \left(-\frac t\lambda\right)^i.
\end{equation*}
Since it is not necessary for the proof of Theorem~\ref{thm:equiv}, we
will prove Lemma~\ref{lem:notpol} in Section~\ref{sec:noequiv}.

Let us from now on assume that $\varphi$ is this solution. Then the
differential equation for $R$ spells
\begin{equation} \label{eq:de:comp}
  [R, \xi] - zqL_ER + zq \frac{L_E \varphi}\varphi R\mu = 0.
\end{equation}

The following lemma implies that the matrix $\tilde R_{\PP^m}(z)
\tilde R_{A_{m + 1}}^{-1}(z\varphi)$ does not have any poles in $t$
and this concludes the proof of Theorem~\ref{thm:equiv}.
\begin{lem} \label{lem:fund}
  For any solution $R(z)$ of \eqref{eq:de:comp} of the form
  \begin{equation*}
    R(z)  = \sum_{i = 0}^\infty (R_{jk}^i) z^i = 1 + O(z),
  \end{equation*}
  for Laurent series $R_{jk}^i$ in $t$, actually all the $R_{jk}^i$
  have to be polynomials.
\end{lem}
\begin{proof}
  The matrices $\xi$ and $\mu$ can be explicitly calculated
  \begin{equation*}
    \xi_{jk} = t \frac{m + 1}{m + 2} \delta_{j, k + 1} (-t)^{\delta_{0, j}}, \qquad
    \mu_{jk} = \frac{2j - m}{2(m + 2)} \delta_{j, k},
  \end{equation*}
  where all indices are understood modulo $(m + 1)$.

  Assume that we have already constructed $R^{i - 1}$ and its entries
  have no negative powers in $t$. Looking at the $z^i$-part of
  \eqref{eq:de:comp} gives expressions for $R_{j(k + 1)}^i\xi_{(k +
    1)k} - \xi_{j(j - 1)} R_{(j - 1)k}^i$ as power series with no
  poles in $t$. From here we see that if we can determine the
  $R_{j0}^i$ as power series with no poles, then the other entries are
  given by
  \begin{equation*}
    R_{jk}^i \equiv (-t)^{\delta_{k > j}} R_{(j - k)0}^i
  \end{equation*}
  modulo terms with no poles in $t$, determined from $R^{i - 1}$. The
  exponent $\delta_{k > j}$ is $1$ for $k > j$ and $0$ otherwise.

  From the $z^{i + 1}$-part of \eqref{eq:de:comp} we then get
  expressions with no poles in $t$ for
  \begin{equation*}
    (m + 1)t \frac{\mathrm dR_{j0}^i}{\mathrm dt} + jR_{j0}^i,
  \end{equation*}
  thus determining all $R_{j0}^i$ but $R_{00}^i$ up to a
  constant. Therefore all the $R_{jk}^i$ are polynomials in $t$.
\end{proof}
\begin{rem}
  The derivation in this section would have worked the same if $q$ was
  any other invertible power series in $t$.
\end{rem}

\section{Higher dimensions} \label{sec:noequiv}

We would like to show that for $m > 1$ there is no pair of function
$\varphi$ and matrix power series $R(z)$, both well-defined in the
limit $\disc \to 0$, such that
\begin{equation} \label{eq:rprod}
  \tilde R_{\PP^m}(z) = R(z) \tilde R_{A_{m + 1}}(z\varphi),
\end{equation}
where again $\tilde R_* = \Psi R_* \Psi^{-1}$. We will need to assume
that that $\varphi$ is well-defined in the Airy limit. Then we can use
the discussion from Section~\ref{sec:equiv} to derive that $\varphi$
is of the form
\begin{equation*}
  \varphi = \lambda + c_0 \lambda^0 + c_{-1} \lambda^{-1} + \dotsb,
\end{equation*}
where the $c_i$ are independent of $\lambda$ and $c_{-1}$ in the Airy
limit becomes a constant multiple of $(t^1)^2$. For the uniqueness of
$\varphi$ we needed Lemma~\ref{lem:notpol}.
\begin{proof}[Proof of Lemma~\ref{lem:notpol}]
  Recall that we have to show that $P := \tilde R_{A_{m + 1}} \xi
  \tilde R_{A_{m + 1}}^{-1}$ is not a polynomial in $z$. From the
  differential equation for $\tilde R_{A_{m + 1}}$ we obtain a
  differential equation for $P$.
  \begin{equation*}
    [P, \xi] = z^2 \frac{\mathrm dP}{\mathrm dz} - z[P, \mu]
  \end{equation*}
  By definition we also have the initial condition $P|_{z = 0} =
  \xi$. Write $P = \xi + zP_1 + z^2P_2 + \dotsb$. The homogeneity
  condition for $\tilde R_{A_{m + 1}}$ implies that the only nonzero
  entries of $P_i$ are at the $(i - 1)$-th diagonal, where by this we
  mean the entries on $j$-th row, $k$-th column such that $k - j
  \equiv i - 1 \pmod{m + 1}$.

  Assume we have shown that $P_i \neq 0$ has a nonzero entry on the $(i
  - 1)$-th diagonal row. Recalling the proof of Lemma~\ref{lem:fund} we
  see that essentially the differences of two subsequent entries in
  the $i$-th diagonal of $P_{i + 1}$ are a multiple of an entry on the
  $(i - 1)$-th diagonal of $P_i$. Since the absolute value of any entry
  of $\mu$ is less than $\frac 12$, all of these multiples are
  nonzero. Therefore it is impossible for all entries on the $i$-th
  diagonal of $P_{i + 1}$ to be zero. The lemma follows by induction.
\end{proof}

To show that there is no suitable intermediate $R$-matrix $R$ it will
be enough to consider the $z^1$-term of \eqref{eq:rprod}. It says
\begin{equation*}
  \tilde r_{\PP^m} = r + \varphi \tilde r_{A_{m + 1}},
\end{equation*}
where $r_*$ stands for the $z^1$-term of $R_*$. Since $\tilde
r_{\PP^m}$ has no negative powers in $\lambda$, the
$\lambda^{-1}$-terms on the right hand side have to cancel. However
the bottom-left coefficient of $\tilde r_{A_{m + 1}}$ has a pole in
the discriminant. Since for $m > 2$ the coefficient $c_{-1}$ cannot be
a multiple of the discriminant for degree reasons, in this case $r$
has to have a pole in the discriminant. Contradiction.

It remains to look at the case $m = 2$. Here it is similarly enough
to show that there is one coefficient in the $R$-matrix with a second
order pole in the discriminant in order to derive a contradiction. We
look at the coefficient $r_{20}$ calculated from the oscillating
integral of Section~\ref{sec:osc:spin}. We need to calculate the
$z^1$-coefficient of the asymptotic expansion of
\begin{equation*}
 \sum_Q \frac 1{\sqrt{2\pi}\Delta} \int\limits_{-\infty}^\infty
  \exp\left(-\frac{x^2}2 - x^3\sqrt{-z}\frac Q{\Delta^{3/2}}- \frac{x^4}4 (-z) \frac 1{\Delta^2}\right) \mathrm dx,
\end{equation*}
where we sum over roots $Q$ of the polynomial defining the singularity
and here $\Delta = 3Q^2 + 2t^2$. Expanding the Gaussian integral we
find the coefficient to be equal to
\begin{equation*}
  -\sum_Q \frac{15}2 \frac{Q^2}{\Delta^4} + \sum_Q \frac 3{\Delta^3}.
\end{equation*}
It is straightforward to check that the first summand equals
\begin{equation*}
  -\frac{15}2 \frac{-2(2t^2)^3 + 27(t^1)^2}{(-4(2t^2)^3 - 27(t^1)^2)^2},
\end{equation*}
whereas the second term has only a first order pole in the
discriminant.

\appendix
\section{Givental's localization calculation}
\label{sec:givenloc}

We want to recall Givental's localization calculation \cite{Gi01a},
which proves that the CohFT from equivariant $\PP^m$ can be obtained
from the trivial theory via a specific $R$-matrix action. We first
recall localization in the space of stable maps to $\PP^m$ in
Section~\ref{sec:locdefs}. Next, in Section~\ref{sec:locgen} we group
the localization contributions according to the dual graph of the
source curve. We collect identities following from the string and
dilaton equation in Section~\ref{sec:SD} before applying them to
finish the computation in Section~\ref{sec:loctofrob}.

\subsection{Localization in the space of stable maps}
\label{sec:locdefs}

Let $T = (\CC^*)^{m + 1}$ act diagonally on $\PP^m$. The equivariant
Chow ring of a point and $\PP^m$ are given by
\begin{align*}
  A^*_{T} (\mathrm{pt})
  \cong& \QQ[\lambda_0, \dotsc, \lambda_m] \\
  A^*_{T} (\PP^m)
  \cong& \QQ[H, \lambda_0, \dotsc, \lambda_m]/\prod_{i = 0}^m (H - \lambda_i),
\end{align*}
where $H$ is a lift of the hyperplane class. Furthermore, let $\eta$ be
the equivariant Poincaré pairing.

There are $m + 1$ fixed points $p_0, \dotsc, p_m$ for the $T$-action
on $\PP^m$. The characters of the action of $T$ on the tangent space
$T_{p_i}\PP^m$ are given by $\lambda_i - \lambda_j$ for $j
\not= i$. Hence the corresponding equivariant Euler class $e_i$
is given by
\begin{equation*}
  e_i = \prod_{j \neq i} (\lambda_i - \lambda_j).
\end{equation*}
The equivariant class $e_i$ also serves as the inverse of the norms
of the equivariant (classical) idempotents
\begin{equation*}
  \phi_i = e_i^{-1} \prod_{j \neq i} (H - \lambda_j).
\end{equation*}

The virtual localization formula \cite{GrPa99} implies that the
virtual fundamental class can be split into a sum
\begin{align*}
  [\Mbar_{g, n}(\PP^m, d)]^{vir}_T
  = \sum_X \iota_{X, *} \frac{[X]^{vir}_T}{e_T(N^{vir}_{X, T})}
\end{align*}
of contributions of fixed loci $X$. Here $N^{vir}_{X, T}$ denotes the
virtual normal bundle of $X$ in $\Mbar_{g, n}(\PP^m, d)$ and $e_T$ the
equivariant Euler class. Because of the denominator, the fixed point
contributions are only defined after localizing by the elements
$\lambda_0, \dotsc, \lambda_m$.  By studying the $\CC^*$-action on
deformations and obstructions of stable maps, $e_T(N^{vir}_{X, T})$
can be computed explicitly.

The fixed loci can be labeled by certain decorated graphs. These
consist of
\begin{itemize}
\item a graph $(V, E)$,
\item an assignment $\zeta: V \to \{p_0, \dotsc, p_m\}$ of fixed points,
\item a genus assignment $g: V \to \ZZ_{\ge 0}$,
\item a degree assignment $d: E \to \ZZ_{> 0}$,
\item an assignment $p: \{1, \dotsc n\} \to V$ of marked points,
\end{itemize}
such that the graph is connected and contains no self-edges, two
adjacent vertices are not assigned to the same fixed point and we have
\begin{equation*}
  g = h^1(\Gamma) + \sum_{v \in V} g(v), \qquad
  d = \sum_{e \in E} d(e).
\end{equation*}
A vertex $v \in V$ is called stable if $2g(v) - 2 + n(v) > 0$, where
$n(v)$ is the number of outgoing edges at $v$.

The fixed locus corresponding to a graph is characterized by the
condition that stable vertices $v \in V$ of the graph correspond to
contracted genus $g(v)$ components of the domain curve, and that edges
$e \in E$ correspond to multiple covers of degree $d(e)$ of the torus
fixed line between two fixed points. Such a fixed locus is isomorphic
to a product of moduli spaces of curves
\begin{equation*}
  \prod_{v \in V} \Mbar_{g(v), n(v)}
\end{equation*}
up to a finite map.

For a fixed locus $X$ corresponding to a given graph the Euler class
$e_T(N^{vir}_{X, T})$ is a product of factors corresponding to the
geometry of the graph
\begin{multline}
  \label{eq:normal}
  e_T(N^{vir}_{X, T}) = \prod_{v\text{, stable}} \frac{e(\mathbb E^*
    \otimes T_{\PP^m, \zeta(v)})}{e_{\zeta(v)}}
  \prod_{\text{nodes}} \frac{e_{\zeta}}{-\psi_1 - \psi_2} \\
  \prod_{\substack{g(v) = 0 \\ n(v) = 1}} (-\psi_v) \prod_e
  \mathrm{Contr}_e.
\end{multline}
In the first product $\mathbb E^*$ denotes the dual of the Hodge
bundle, $T_{\PP^m, \zeta(v)}$ is the tangent space of $\PP^m$ at
$\zeta(v)$, and all bundles and Euler classes should be considered
equivariantly. The second product is over nodes forced onto the domain
curve by the graph. They correspond to stable vertices together with
an outgoing edge, or vertices $v$ of genus $0$ with $n(v) = 2$. With
$\psi_1$ and $\psi_2$ we denote the (equivariant) cotangent line
classes at the two sides of the node. For example, the equivariant
cotangent line class $\psi$ at a fixed point $p_i$ on a line mapped
with degree $d$ to a fixed line is more explicitly given by
\begin{equation*}
  -\psi = \frac{\lambda_j - \lambda_i}d,
\end{equation*}
where $p_j$ is the other fixed point on the fixed line. The explicit
expressions for the terms in the second line of \eqref{eq:normal} can
be found in \cite{GrPa99}, but will play no role for us. It is only
important that they only depend on local data.

\subsection{General procedure}
\label{sec:locgen}

We set $W$ to be $A^*_{T} (\PP^m)$ with all equivariant parameters
localized. For $v_1, \dotsc, v_n \in W$ the (full) CohFT
$\Omega_{g, n}$ from equivariant $\PP^m$ is defined by
\begin{multline}
  \label{eq:PmfullCohFT}
  \Omega_{g, n}^p(v_1, \dotsc, v_n) \\
  = \sum_{d, k = 0}^\infty \frac{q^d}{k!}
  \varepsilon_* \pi_*\left(\prod_{i = 1}^n \ev_i^*(v_i) \prod_{i = n + 1}^{n + k}
    \ev_i^*(p) \cap [\Mbar_{g, n + k}(\PP^m, d)]^{vir}\right),
\end{multline}
where $p$ is a point on the formal Frobenius manifold, $\varepsilon$
forgets the last $k$ markings and $\pi$ forgets the map. We want to
calculate the push-forward via virtual localization. In the end we
will arrive at the formula of the $R$-matrix action as described in
Section~\ref{sec:Rmatrixaction}. In the following we will
systematically suppress the dependence on $p$ in the notation.

We start by remarking that for each localization graph for
\eqref{eq:PmfullCohFT} there exists a dual graph of $\Mbar_{g, n}$
corresponding to the topological type of the stabilization of a
generic source curve of that locus. What gets contracted under the
stabilization maps are trees of rational curves. There are three types
of these unstable trees:
\begin{enumerate}
\item those which contain one of the $n$ markings and are connected to
  a stable component,
\item those which are connected to two stable components and contain
  none of the $n$ markings and
\item those which are connected to one stable
component but contain none of the $n$ markings.
\end{enumerate}
These give rise to series of localization contributions and we want to
record those, using the fact that they already occur in genus 0.

Let $W'$ be an abstract free module over the same base ring as $W$
with a basis $w_0, \dotsc, w_m$ labeled by the fixed points of the
$T$-action on $\PP^m$. The type 1 contributions are recorded by
\begin{equation*}
  \tilde R^{-1} = \sum_i \tilde R_i^{-1} w_i \in \Hom(W, W')[\![z]\!],
\end{equation*}
the homomorphism valued power series such that
\begin{equation*}
  \tilde R_i^{-1}(v) = \eta(e_i\phi_i, v) + \sum_{d, k = 0}^\infty \frac{q^d}{k!} \sum_{\Gamma \in G_{d, k, i}^1} \frac 1{\Aut(\Gamma)} \Contr_\Gamma(v)
\end{equation*}
where $G_{d, k, i}^1$ is the set of localization graphs for
$\Mbar_{0, 2 + k}(\PP^m, d)$ such that the first marking is at a
valence 2 vertex at fixed point $i$ and $\Contr_\Gamma(v)$ is the
contribution for graph $\Gamma$ for the integral
\begin{equation*}
  \int_{\Mbar_{0, 2 + k}(\PP^m, d)} \frac{e_i}{-z - \psi_1} \ev_2^*(v) \prod_{l = 3}^{2 + k} \ev_l^*(p).
\end{equation*}
We define the integral in the case $(d, k) = (0, 0)$ to be zero and
will do likewise for other integrals over non-existing moduli spaces.

The type 2 contributions are recorded by the bivector
\begin{equation*}
  \tilde V = \sum_i \tilde V^{ij} w_i \otimes w_j \in W'^{\otimes 2} [\![z, w]\!]
\end{equation*}
which is defined by
\begin{equation*}
  V^{ij} = \sum_{d, k = 0}^\infty \frac{q^d}{k!} \sum_{\Gamma \in G_{d, k, i, j}^2} \frac 1{\Aut(\Gamma)} \Contr_\Gamma,
\end{equation*}
where $G_{d, k, i, j}^2$ is the set of localization graphs for
$\Mbar_{0, 2 + k}(\PP^m, d)$ such that the first and second marking
are at valence 2 vertices at fixed points $i$ and $j$, respectively,
and $\Contr_\Gamma$ is the contribution for graph $\Gamma$ for the
integral
\begin{equation*}
  \int_{\Mbar_{0, 2 + k}(\PP^m, d)} \frac{e_ie_j}{(-z - \psi_1)(-w - \psi_2)} \prod_{l = 3}^{2 + k} \ev_l^*(p).
\end{equation*}
Finally, the type 3 contribution is a vector
\begin{equation*}
  \tilde T = \sum_i \tilde T_i w_i \in W'[\![z]\!]
\end{equation*}
which is defined by
\begin{equation*}
  \tilde T_i = p + \sum_{d, k = 0}^\infty \frac{q^d}{k!} \sum_{\Gamma \in G_{d, k, i}^3} \frac 1{\Aut(\Gamma)} \Contr_\Gamma
\end{equation*}
where $G_{d, k, i}^3$ is the set of localization graphs for
$\Mbar_{0, 1 + k}(\PP^m, d)$ such that the first marking is at a
valence 2 vertex at fixed point $i$ and $\Contr_\Gamma(v)$ is the
contribution for graph $\Gamma$ for the integral
\begin{equation*}
  \int_{\Mbar_{0, 1 + k}(\PP^m, d)} \frac{e_i}{-z - \psi} \prod_{l = 2}^{1 + k} \ev_l^*(p).
\end{equation*}

With these contributions we can write the CohFT already in a form
quite similar to the reconstruction formula. Let $\omega_{g, n}$ be
the $n$-form on $W'$ which vanishes if $w_i$ and $w_j$ for $i \neq j$
are inputs, which satisfies
\begin{equation*}
  \omega_{g, n}(w_i, \dotsc, w_i) = \frac{e(\mathbb E^* \otimes T_{\PP^m, p_i})}{e_i} = e_i^{g - 1} \prod_{j \neq i} c_{\lambda_j - \lambda_i}(\mathbb E)
\end{equation*}
and which is for $n = 0$ defined similarly as in
Example~\ref{ex:trivial}. We have
\begin{equation}
  \label{eq:locresA}
  \Omega_{g, n}^p(v_1, \dotsc, v_n) = \sum_{\Gamma} \frac 1{\Aut(\Gamma)} \xi_*\left(\prod_v \sum_{k = 0}^\infty \frac 1{k!} \varepsilon_* \omega_{g_v, n_v + k}(\dots)\right),
\end{equation}
where we put
\begin{enumerate}
\item $\tilde R^{-1}(\psi)(v_i)$ into the argument corresponding to marking $i$,
\item a half of $\tilde V(\psi_1, \psi_2)$ into an argument
  corresponding to a node and
\item $\tilde T(\psi)$ into all additional arguments.
\end{enumerate}

We will still need to apply the string and dilaton equation in order
to make $\tilde T(\psi)$ to be a multiple of $\psi^2$, like the
corresponding series in the reconstruction, and then relate the series
to the $R$-matrix.

\subsection{String and Dilaton Equation}
\label{sec:SD}

We want to use the string and dilaton equation to bring a series
\begin{equation}
  \label{eq:SDstart}
  \sum_{k = 0}^\infty \frac 1{k!}
  \varepsilon_*\left( \prod_{i = 1}^n \frac 1{-x_i - \psi_i} \prod_{i =
      n + 1}^{n + k}
    Q(\psi_i)\right),
\end{equation}
where $\varepsilon: \Mbar_{g, n + k} \to \Mbar_{g, n}$ is the
forgetful map and $Q = Q_0 + zQ_1 + z^2Q_2 + \dotsb$ is a formal
series, into a canonical form.

By the string equation, \eqref{eq:SDstart} is annihilated by
\begin{equation*}
  \mathcal L' = \mathcal L + \sum_{i = 1}^n \frac 1{x_i},
\end{equation*}
where $\mathcal L$ is the string operator
\begin{equation*}
  \mathcal L
  = \frac\partial{\partial Q_0} - Q_1 \frac\partial{\partial Q_0} - Q_2 \frac\partial{\partial Q_1} - Q_3 \frac\partial{\partial Q_2} - \dotsb.
\end{equation*}
Moving along the string flow for some time $-u$, i.e. applying
$e^{t\mathcal L'}|_{t = -u}$, to \eqref{eq:SDstart} gives
\begin{equation*}
  \sum_{k = 0}^\infty \frac 1{k!}
  \varepsilon_*\left( \prod_{i = 1}^n \frac{e^{-\frac u{x_i}}}{-x_i - \psi_i} \prod_{i =
      n + 1}^{n + k}
    Q'(\psi_i)\right),
\end{equation*}
for a new formal series $Q' = Q'_0 + zQ'_1 + z^2Q'_2 + \dotsb$. In the
case that
\begin{equation*}
  u
  = \sum_{k = 1}^\infty \frac 1{k!} \int\limits_{\Mbar_{0, 2 + k}} \prod_{i = 3}^{2 + k} Q(\psi_i),
\end{equation*}
which we will assume from now on, the new series $Q'$ will satisfy
$Q'_0 = 0$ since by the string equation $\mathcal Lu = 1$ and
therefore applying $e^{t\mathcal L}|_{t = -u}$ to $u$ gives on the one
hand zero and on the other hand the definition of $u$ with $Q$
replaced by $Q'$, and for dimension reasons this is a nonzero multiple
of $Q'_0$.

Next, by applying the dilaton equation we can remove the linear part
from the series $Q'_0$
\begin{multline}
  \label{eq:SDstable}
  \sum_{k = 0}^\infty \frac 1{k!}
  \varepsilon_*\left( \prod_{i = 1}^n \frac 1{-x_i - \psi_i} \prod_{i =
      n + 1}^{n + k}
    Q(\psi_i)\right) \\
  = \sum_{k = 0}^\infty \frac{\Delta^{\frac{2g - 2 + n + k}2}}{k!}
  \varepsilon_*\left( \prod_{i = 1}^n \frac{e^{-\frac u{x_i}}}{-x_i - \psi_i} \prod_{i =
      n + 1}^{n + k}
    Q''(\psi_i)\right),
\end{multline}
where $Q'' = Q' - Q'_1z$ and
\begin{equation*}
  \Delta^{\frac 12} = (1 - Q'_1)^{-1} = \sum_{k = 0}^\infty \frac 1{k!} \int\limits_{\Mbar_{0, 3 + k}} \prod_{i = 4}^{3 + k} Q(\psi_i).
\end{equation*}

We will also need identities in the degenerate cases $(g, n) = (0, 2)$
and $(g, n) = (0, 1)$. In the first case, there is the identity
\begin{equation}
  \label{eq:sflow2}
  \frac 1{-z - w} + \sum_{k = 1}^\infty \frac 1{k!} \int\limits_{\Mbar_{0, 2 + k}} \frac 1{-z - \psi_1} \frac 1{-w - \psi_2} \prod_{i = 3}^{2 + k} Q(\psi_i)
  = \frac{e^{-u/z + -u/w}}{-z - w}.
\end{equation}
In order to see that \eqref{eq:sflow2} is true, we use that the left
hand side is annihilated by $\mathcal L + \frac 1z + \frac 1w$ in order
to move from $Q$ to $Q'$ via the string flow and notice that there all
of the integrals vanish for dimension reasons. Similarly, there is the
identity
\begin{multline}
  \label{eq:sflowT}
  1 - \frac{Q(z)}{z} - \frac 1z \sum_{k = 2}^\infty \frac 1{k!} \int\limits_{\Mbar_{0, 1 + k}} \prod_{i = 2}^{1 + k} Q(\psi_i) \\
  = e^{-u/z} \left(1 - \frac{Q'(z)}{z}\right) = e^{-u/z} \left(\Delta^{-\frac 12} - \frac{Q''(z)}{z}\right),
\end{multline}
which can be proven like the previous identity by using that the left
hand side is annihilated by $\mathcal L + \frac 1z$.

We define the functions $u_i$ and $(\Delta_i/e_i)^{\frac 12}$ for
$i \in \{0, \dotsc, m\}$ to be the $u$ and $\Delta^{\frac 12}$ at the
points $Q = \tilde T_i$ from the previous section.

\subsection{Expressing localization series in terms of Frobenius
  structures}
\label{sec:loctofrob}

We apply \eqref{eq:SDstable} to \eqref{eq:locresA} and obtain
\begin{equation}
  \label{eq:locresB}
  \Omega_{g, n}^p(v_1, \dotsc, v_n) = \sum_{\Gamma} \frac 1{\Aut(\Gamma)} \xi_*\left(\prod_v \sum_{k = 0}^\infty \frac 1{k!} \varepsilon_* \omega'_{g_v, n_v + k}(\dots)\right),
\end{equation}
where we put
\begin{enumerate}
\item $R^{-1}(\psi)(v_i)$ into the argument corresponding to marking $i$,
\item a half of $V(\psi_1, \psi_2)$ into an argument
  corresponding to a node and
\item $T(\psi)$ into all additional arguments.
\end{enumerate}
Here $R^{-1}$, $V$ and $T$ are defined exactly as $\tilde R^{-1}$,
$\tilde V$ and $\tilde T$ but with the replacement
\begin{equation*}
  \frac{e_i}{-x - \psi} \rightsquigarrow \frac{e_i e^{-\frac{u_i}x}}{-x - \psi}
\end{equation*}
made at the factors we put at the ends of the trees. The form
$\omega'_{g, n}$ satisfies
\begin{equation*}
  \omega'_{g, n}(w_i, \dotsc, w_i) = \Delta_i^{\frac{2g - 2 + n}2} e_i^{-\frac n2} \prod_{j \neq i} c_{\lambda_j - \lambda_i}(\mathbb E).
\end{equation*}

We now want to compute $R^{-1}$, $V$ and $T$ in terms of the
homomorphism valued series $S^{-1}(z) \in \Hom(W, W')[\![z]\!]$ with
$w_i$-component
\begin{multline*}
  S_i^{-1}(z) = \langle \frac{e_i\phi_i}{-z - \psi}, -\rangle \\
  := \eta(e_i\phi_i, -) + \sum_{d, k = 0}^\infty \frac{q^d}{k!} \int\limits_{\Mbar_{0, 2 + k}(\PP^m, d)} \frac{\ev_1^*(e_i\phi_i)}{-z - \psi_1} \ev_2^*(-) \prod_{j = 3}^{2 + k} \ev_j^*(p).
\end{multline*}

We start by computing $S^{-1}$ via localization. Using that in genus
zero the Hodge bundle is trivial we find that at the vertex with the
first marking we need to compute integrals exactly as in
\eqref{eq:sflow2}, where the first summand stands for the case that
the vertex is unstable and the $k$-summand stands for the case that
the vertex is stable with $k$ trees of type 3 and one tree of type
1. Applying \eqref{eq:sflow2} we obtain
\begin{equation*}
  S_i^{-1}(z) = e^{-\frac{u_i}z} R_i^{-1}(z).
\end{equation*}

Using the short-hand notation
\begin{multline*}
  \left\langle \frac{v_1}{x_1 - \psi}, \frac{v_2}{x_2 - \psi}, \frac{v_3}{x_3 - \psi}\right\rangle \\
  := \sum_{d, k = 0}^\infty \frac{q^d}{k!}
  \int\limits_{\Mbar_{0, 3 + k}(\PP^m, d)}
  \frac{\ev_1^* v_1}{x_1 - \psi_1} \frac{\ev_2^* v_2}{x_2 - \psi_2} \frac{\ev_3^* v_3}{x_3 - \psi_3} \prod_{i = 4}^{3 + k} \ev_i^* (p)
\end{multline*}
for genus zero Gromov-Witten invariants and applying the string
equation, we can also write
\begin{equation*}
  S_i^{-1}(z) = -\frac 1z \langle \frac{e_i\phi_i}{-z - \psi}, \mathbf 1, -\rangle.
\end{equation*}
We have by the identity axiom and WDVV equation
\begin{multline*}
  \langle \frac{e_i\phi_i}{-z - \psi}, \frac{e_j\phi_j}{-w - \psi}, \mathbf 1\rangle = \langle \frac{e_i\phi_i}{-z - \psi}, \frac{e_j\phi_j}{-w - \psi}, \bullet\rangle \langle \bullet, \mathbf 1, \mathbf 1\rangle \\
  = \langle \frac{e_i\phi_i}{-z - \psi}, \mathbf 1, \bullet\rangle \langle \bullet, \mathbf 1, \frac{e_j\phi_j}{-w - \psi}\rangle,
\end{multline*}
where in the latter two expressions the $\bullet$ should be filled with
$\eta^{-1}$, so that
\begin{multline}
  \label{eq:SetaS}
  \frac{S_i^{-1}(z) \eta^{-1} S_j^{-1}(w)^t}{-z - w} \\
  = \frac{\eta(e_i\phi_i, e_j\phi_j)}{-z - w} + \sum_{d, k = 0}^\infty \frac{q^d}{k!} \int\limits_{\Mbar_{0, 2 + k}(\PP^m, d)} \frac{\ev_1^*(e_i\phi_i)}{-z - \psi_1} \frac{\ev_2^*(e_j\phi_j)}{-w - \psi_2} \prod_{l = 3}^{2 + k} \ev_l^*(p).
\end{multline}
We compute the right hand side via localization. There are two cases
in the localization depending on whether the first and second marking
are at the same or a different vertex. In the first case we apply
\eqref{eq:sflow2} at this common vertex and obtain the total
contribution
\begin{equation*}
  \frac{e_i\delta_{ij} e^{-\frac{u_i}z - \frac{u_j}w}}{-z - w},
\end{equation*}
which includes the unstable summand. In the other case, we apply
\eqref{eq:sflow2} at the two vertices and obtain
\begin{equation*}
  e^{-\frac{u_i}z - \frac{u_j}w} V^{ij}(z, w).
\end{equation*}
So all together
\begin{equation*}
  V^{ij}(z, w) = \frac{R_i^{-1}(z)\eta^{-1}R_j^{-1}(w)^t - e_i\delta_{ij}}{-z - w}.
\end{equation*}

Finally we express $T$ in terms of $R$ by computing
\begin{equation*}
  S_i^{-1}(z)\mathbf 1 = e_i - \frac 1z\sum_{d, k = 0}^\infty \frac{q^d}{k!} \int\limits_{\Mbar_{0, 1 + k}(\PP^m, d)} \frac{\ev_1^*(e_i\phi_i)}{-z - \psi_1} \prod_{j = 2}^{1 + k} \ev_j^*(p)
\end{equation*}
via localization. Applying \eqref{eq:sflowT} at the first marking we
find that
\begin{equation*}
  S_i^{-1}(z)\mathbf 1 = e^{-\frac{u_i}z} \left(\Delta_i^{-\frac 12} e_i^{\frac 12} - \frac{T_i(z)}z\right).
\end{equation*}
So
\begin{equation*}
  T(z) = z\left(\sum_i \Delta_i^{-\frac 12} e_i^{\frac 12}w_i - R^{-1}(z) \mathbf 1\right).
\end{equation*}

By \eqref{eq:locresB} the underlying TQFT of $\Omega^p_{g, 0}$ is
given by
\begin{equation*}
  \sum_i \Delta_i^{g - 1}.
\end{equation*}
This implies that the $\Delta_i$ need to be the inverses of the norms
of the idempotents for the quantum product of equivariant $\PP^m$
(because these are pairwise different). Since $\tilde T$ vanishes at
$(p, q) = 0$, $\Delta_i$ is $e_i^{-1} $ at $(p, q) \to 0$. Therefore
we can identify $W'$ with $W$ by mapping $w_i$ to
$\sqrt{\Delta_i/e_i}$ times the idempotent element which coincides
with $\phi_i$ at $(p, q) = 0$. The previous results then say exactly
that $\Omega^p$ is obtained from the CohFT $\omega'$ by the action of
the $R$-matrix $R$. In turn, Example~\ref{ex:mumford} implies that
$\omega'$ is obtained from the TQFT by the action of an $R$-matrix
which is diagonal in the basis of idempotents and has entries
\begin{equation*}
  \exp\left(\sum_{i = 1}^\infty \frac{B_{2i}}{2i(2i - 1)} \sum_{j \neq i} \left(\frac z{\lambda_j - \lambda_i}\right)^{2i - 1}\right).
\end{equation*}

We still need to check that $R$ satisfies the quantum differential
equation and has the correct limit \eqref{eq:limit} as $q \to 0$.  By
considering \eqref{eq:SetaS} as $w + z \to 0$ we see that $S^{-1}(z)$
satisfies the symplectic condition, i.e. its inverse $S(z)$ is the
adjoint with respect to $\eta$ of $S^{-1}(-z)$. More explicitly the
evaluation of $S(z)$ at the $i$th normalized idempotent is the vector
\begin{equation*}
  \left\langle \frac{\sqrt{e_i}\phi_i}{z - \psi}, \eta^{-1}\right\rangle.
\end{equation*}
By the genus 0 topological recursion relations for any flat vector
field $X$
\begin{equation*}
  z\left\langle X, \frac{\sqrt{e_i}\phi_i}{z - \psi}, \eta^{-1}\right\rangle = \langle X, \eta^{-1}, \bullet\rangle \left\langle \bullet, \frac{\sqrt{e_i}\phi_i}{z - \psi}\right\rangle,
\end{equation*}
where again $\eta^{-1}$ should be inserted at the $\bullet$. Therefore
$S$ satisfies the quantum differential equation
\begin{equation*}
  z XS(z) = X \star S(z),
\end{equation*}
where on the left hand side the action of vector fields and the right
hand side quantum multiplication is used.

At $q = 0$, we can check that $R$ becomes the identity matrix and
therefore the $R$-matrix of $\Omega^p$ becomes the $R$-matrix of the
CohFT $\omega'$, which has the correct limit \eqref{eq:limit}.

\printbibliography
\addcontentsline{toc}{section}{References}

\vspace{+8 pt}
\noindent
Departement Mathematik \\
ETH Zürich \\
felix.janda@math.ethz.ch

\end{document}